\documentclass[11pt,a4paper]{amsart}
\usepackage[utf8]{inputenc}
\usepackage{amsaddr}
\usepackage[english]{babel}
\usepackage{latexsym,amssymb}
\usepackage{xcolor}
\usepackage[pdftex]{graphicx}
\usepackage{tikz}
\usepackage{amsmath,amsthm,amsfonts,amssymb,bbm}
\usepackage{graphicx,psfrag,subfigure,float,color}
\usepackage{cite}
\usepackage{graphicx}
\usetikzlibrary{arrows}
\usepgflibrary{snakes}
\usetikzlibrary{arrows,snakes,backgrounds}
\DeclareGraphicsExtensions{.jpg,.pdf,.pdftex,.eps}

\newcommand{\Pb}{\mathbb{P}}

\newcommand{\dx}{\mathrm{d}}
\newcommand{\R}{\mathbb{R}}

\newcommand{\Z}{\mathbb{Z}}

\newcommand{\GUE}{\mathrm{GUE}}

\newtheorem{tthm}{Theorem}

\newtheorem{prop}{Proposition}[section]

\newtheorem{lem}[prop]{Lemma}
\newtheorem{defin}[prop]{Definition}

\newtheorem{rem}[prop]{Remark}

\usepackage{graphicx}
\usepackage{amssymb}
\usepackage{epstopdf}
\usepackage{nicefrac}
\author{
Peter Nejjar
}
\address{IST Austria, 3400 Klosterneuburg, Austria. }
\thanks{Research supported by ERC Advanced Grant No. 338804 and ERC Starting Grant No. 716117.}
\email{peter.nejjar@ist.ac.at}
\begin{document}

\vfill
\begin{abstract}
We consider the asymmetric simple exclusion process (ASEP) on $\Z$ with an initial data such that in the large time particle density $\rho(\cdot)$ a discontinuity at the origin is created, where the value of $\rho$ jumps from zero to one, but 
$\rho(-\varepsilon),1-\rho(\varepsilon) >0 $ for any $\varepsilon>0$. We consider the position of a particle $x_{M}$ macroscopically located at the discontinuity, and show that its limit law has a cutoff under $t^{1/2}$ scaling. 
Inside the discontinuity region, we show that a discrete product limit law arises, which bounds from above the limiting  fluctuations of $x_{M}$ in the general ASEP, and equals them  in the  totally ASEP.
\end{abstract}
\vfill
\title{Cutoff and discrete product structure in ASEP}

\date{}

\maketitle 
\section{Introduction}
We consider the asymmetric simple exclusion process (ASEP) on $\Z$. In this model, particles move in $\Z$  and there is at most one particle per site. Each particle waits (independently of all other particles)  an exponential time (with parameter $1$) to attempt to move one unit step, which is a step   to the right with probability $p>1/2$, and a step 
to the left with probability $q=1-p$. The attempted jump is succesful iff the target site is empty (exclusion constraint). 
ASEP is a continuous time Markov process  with state space $X=\{0,1\}^{\Z}$ and we denote by $\eta_{\ell}\in X$ the particle configuration at time $\ell$, see \cite{Li85b} for the rigorous construction of ASEP.
If $p=1$ we speak of the totally ASEP (TASEP), whereas if $p\in(1/2,1)$ we speak of the partially ASEP (PASEP).

The hydrodynamical behavior of ASEP is well established: For ASEP starting from $\eta_{0}\in X$, assume that
 \begin{equation}  \label{partdens}
\lim_{N \to \infty}\frac{1}{N}\sum_{i \in \Z}\delta_{\frac{i}{N}}\eta_{0}(i)=\rho_{0}(\xi)\dx \xi,
\end{equation}
where $\delta_{i/N}$ is the dirac measure at $i/N$ and the convergence is in the sense of vague convergence of measures.
Then the large time density of the ASEP is given by 
 \begin{equation}  \label{partdens2}
\lim_{N \to \infty}\frac{1}{N}\sum_{i \in \Z}\delta_{\frac{i}{N}}\eta_{N}(i)=\rho(\xi)\dx \xi,
\end{equation}
where $\rho(\xi)$ is the unique entropy solution of the Burgers equation with initial data $\rho_{0}$.
Given an initial data $\eta_{0}$ we can assign a label (an integer) to each particle, and we denote by $x_{M}(t)$ the position at time $t$ of the particle with label $M$.
It has been an question of great interest to study the fluctuations  of $x_{M}(t)$ around its macroscopic limit given by \eqref{partdens2}, for various choices of $M$ and $\eta_{0}$, as well as closely related quantities 
like the height function and current of ASEP. TASEP turned out to be more accesible than PASEP, and it has been  of great interest to generalize results from TASEP to PASEP.

An important breakthrough  in this regard has been the paper \cite{TW08b}, where the authors consider ASEP with step initial data $\eta^{\mathrm{step}}=\mathbf{1}_{\{i\leq 0\}}$\footnote{In fact, they consider the initial data $\eta=\mathbf{1}_{\{i\geq 0\}}$ with particles having a drift to the left, which is equivalent.}. The limiting particle density \eqref{partdens2} then has a region of decreasing density, and  \cite{TW08b} shows that for a particle located in this region, its fluctuations around its macroscopic position are of order $t^{1/3}$ and given by the Tracy-Widom $F_{\GUE}$ distribution (Theorem 3 in \cite{TW08b}). This is commonly called  Kardar-Parisi-Zhang (KPZ) fluctuation behavior, see \cite{Cor11} for a review. For TASEP, this result had been shown earlier in \cite{Jo00b}, Theorem 1.6. The authors of  \cite{TW08b} also obtained the limit law of the rescaled position of the particle initially at position $-M$ ($M$ fixed), see Theorem \ref{convthm} below.
The results of  \cite{TW08b} were later extended to so-called (generalized) step Bernoulli initial data \cite{TW09b}, \cite{AB16b}. For stationary ASEP (where $\eta_{0}(i),i\in \Z$ are i.i.d. Bernoullis, with $\eta_{0}$ the initial configuration), \cite{A18} showed that the current fluctuations along the characteristics converge to the Baik-Rains distribution, again generalizing a result known for TASEP (\cite{FS06v}) to the general  ASEP. Considerable effort has also been devoted to (half-) flat initial data \cite{OQR16}, \cite{OQR18}, 
which  again are already understood for TASEP \cite{BFS07}, \cite{BFPS06}.

 An important feature of the macroscopic density $\rho$  is that it can have discontinuties (shocks). For example, if $\rho_{0}(\xi)=\lambda \mathbf{1}_{\xi\geq0}+\mu \mathbf{1}_{\xi<0},1 \geq \lambda>\mu\geq 0,$ then $\rho(\xi+(p-q)(1-\lambda-\mu))=\rho_{0}(\xi)$. In this example, and if  this density profile has been created by a random bernoulli initial data, it is known for the ASEP that at the shock,  the gaussian fluctuations  in the initial data supersede those of ASEP itself (see \cite{Li99}, Part III for a presentation of the relevant results). For shocks created by deterministic initial data, fluctuation results for particles at the shock position are to the best of our knowledge so far restricted to TASEP.   In \cite{BFS09} Proposition 1, a shock is studied which has a cutoff on the $t^{1/2}$ scaling.  In \cite{FN14}, the authors consider various TASEP initial data which create shocks between two regions of constant density (the aforementioned case), two regions of decreasing density and between a decreasing and constant density region (see Corollaries 2.5, 2.6, 2.7 in \cite{FN14}). The main result of \cite{FN14} is that the limiting  fluctuations of a particle at the shock are
 of product form. This product form is due to an \textit{asymptotic independence} the authors found by mapping TASEP to last passage percolation. The concept of asymptotic independence has then  appeared in subsequent works \cite{FN15}, \cite{Nej18},\cite{FGN18},\cite{FN17} on TASEP.
 
 Furthermore,  the paper \cite{F18} proved a generalization of Proposition 1 of \cite{BFS09}, where again a cutoff 
 under $t^{1/2}$ scaling is observed.  Additionally,  \cite{F18}  does not utilize the connection of TASEP to last passage percolation, but rather a coupling within TASEP that can be seen as an application of the coupling provided in Lemma 2.1 of \cite{Sep98c}.
 We use the same coupling (see \eqref{coupling}) , however for the general ASEP this provides only  an upper bound rather than an  identity (see \eqref{Y}, \eqref{TY}). Next to this coupling, a key probabilistic tool we use is slow-decorrelation \cite{Fer08}.
 
To the best of our knowledge, the aforementioned    product limit laws and cutoffs at shocks have so far not been observed  in PASEP, and this is the main contribution of this paper.  We  consider a situation where both a cutoff  under $t^{1/2}$ scaling  and a product form appears, the latter being an upper bound to the particle fluctuations in the general ASEP, which becomes an identity in TASEP, and  our  main result Theorem \ref{main} is new  even for TASEP.
 
 Specifically, 
 we will consider the initial data
\begin{equation}\label{IC}
x_{n}(0)=
\begin{cases}
-n-\lfloor (p-q)t\rfloor  \quad &\mathrm{for} \,  n \geq 1 \\
-n\quad &\mathrm{for}\, -\lfloor (p-q)t\rfloor \leq n \leq 0.\\
\end{cases}
\end{equation}
and denote by $(\eta_{\ell})_{\ell\geq 0}$ the ASEP started from this initial data. To be clear, the initial data \eqref{IC} is to be understood that for each fixed $t\geq 0 $,  we start ASEP with this initial data, and let it run up to time $t,$ and study the position of particles at time  $t$. In particular we have a sequence of initial configurations.

 This creates what we call a massive shock in the density $\rho,$  where its value jumps from $0$ to $1$, see Figure \ref{curvdens}.

\begin{figure}
 \begin{center}
   \begin{tikzpicture}
       \draw (0.4,0.7) node[anchor=south]{\small{$1$}};
  \draw [very thick, ->] (0,-0.5) -- (0,2.5);
    \draw (-0.1,2) node[anchor=east] {\small{$\rho_{0}(\xi)$}};
    \draw[red,thick ] (-2,1.3) -- (-1,1.3);
      \draw[thick,red] (0,1.3) -- (1,1.3);
\draw[very thick] (-1,-0.1)--(-1,0.1)  ;
  \draw (-1,-0.6) node[anchor=south]{\small{$q-p$}};
  \draw[very thick] (1,-0.1)--(1,0.1)  ;
  \draw (1,-0.6) node[anchor=south]{\small{$p-q$}};
\filldraw(0,1.3) circle(0.08cm);
   \draw [very thick, ->] (-2,0) -- (2,0) node[below=4pt] {\small{$\xi$}};

\begin{scope}[xshift=7cm]
 \draw [very thick,->] (-2,0) -- (2,0) node[below=4pt] {\small{$\xi$}};
  \draw [very thick,->] (0,-0.5) -- (0,2.5);
    \draw (0.3,1.3) node[anchor=south]{\small{$1$}};
    \draw[red,thick] (-2,1.3) -- (0,0);
    \draw[red,thick] (0,1.3)--(2,0);
    \draw[very thick] (-1,-0.1) -- (-1,0.1);
       \draw (-0.8,-0.6) node[anchor=south]{\small{$q-p$}};
    \draw[very thick] (1,-0.1) -- (1,0.1);
       \draw (1.2,-0.6) node[anchor=south]{\small{$p-q$}};
\filldraw(0,1.3) circle(0.08cm); 
    \draw (-0.1,2) node[anchor=east] {\small{$\rho(\xi)$}};
   \end{scope}
   \end{tikzpicture}    \end{center} \caption{Left: The  initial particle density $\rho_{0}$ of  \eqref{partdens} for the initial configuration \eqref{IC}.
     Right: The large time particle density $\rho$ of \eqref{partdens2} for the same initial configuration. At the origin, $\rho$ jumps from $0$ to $1$, and $\rho(-\varepsilon),1-\rho(\varepsilon)>0$ for any $\varepsilon>0$. }\label{curvdens}
\end{figure}
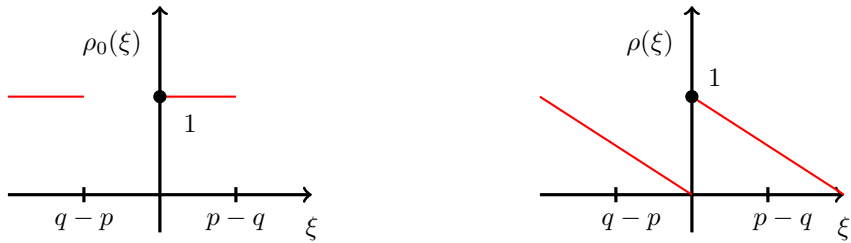

Before coming to the main result of this paper (Theorem \ref{main}), let us define the probability distribution functions which appear in it.

\begin{defin}[ \cite{TW08b},\cite{GW90}] Let $s \in \R,M \in \Z_{\geq 1}$. We define for $p\in (1/2,1)$  
\begin{equation} \label{def1}
F_{M,p}(s)=\frac{1}{2 \pi i} \oint_{| \lambda|> (q/p)^{-M+1}} \frac{\dx \lambda}{\lambda} \frac{\det(I-\lambda K)}{\prod_{k=0}^{M-1}(1-\lambda (q/p)^{k})}
\end{equation}
where $K= \hat{K}\mathbf{1}_{(-s,\infty)}$  and 
$\hat{K}(z,z^{\prime})=\frac{p}{\sqrt{2 \pi}}e^{-(p^{2}+q^{2})(z^{2}+z^{\prime 2})/4+pqzz^{\prime}}$. For $p=1$, we define
\begin{equation}
F_{M,1}(s)=\Pb\left(\sup_{0=t_{0}<\cdots <t_{M}=1}\sum_{i=0}^{M-1}[B_{i}(t_{i+1})-B_{i}(t_{i})]\leq s    \right),
\end{equation}
where $B_{i},i=0,\ldots,M-1$ are independent standard brownian motions.\end{defin}
 It follows from \cite{Bar01}, Theorem 0.7 that $F_{M,1}$ equals the distribution function of the largest eigenvalue of a $M \times M\, \GUE$  matrix. What is important to us here is that $F_{M,p}$ arises as limit law in ASEP,
 a result we cite in Theorem \ref{convthm} below.

Next we come to the aforementioned coupling. We define the  initial data
\begin{equation}
x_{n}^{A}(0)=
-n-\lfloor (p-q)t\rfloor  \quad \mathrm{for} \quad   n \geq 1 
\end{equation}
and
\begin{equation}
x_{n}^{B}(0)=
-n  \quad \mathrm{for} \quad   n \geq  -\lfloor (p-q)t\rfloor,
\end{equation}
and denote by $ (\eta_{\ell}^{A})_{\ell  \geq 0},(\eta_{\ell}^{B})_{\ell  \geq 0}$ the ASEPs started from these initial data.
We set for $n\geq 1$
\begin{equation}\label{coupling}
Y_{n}(t)=\min\{x_{n}^{A}(t),x_{n}^{B}(t)\}.
\end{equation}
In TASEP, we have

\begin{equation}\label{TY}
Y_{n}(t)=x_{n}(t),
\end{equation}
whereas for ASEP we have 
\begin{equation}\label{Y}
Y_{n}(t)\geq x_{n}(t).
\end{equation}
The following  is the main result of this paper. By \eqref{Y} and \eqref{TY}, \eqref{prodform} equals the limiting behavior of $x_{M}$ for TASEP, and provides an upper bound for it in the general ASEP.
\begin{tthm}\label{main}
We have that 
\begin{equation}\label{cutoff}
\lim_{t \to \infty}\Pb(x_{M}(t)\geq -(p-q)s t^{1/2})=
\begin{cases}
0  \quad &\mathrm{for} \,  s <0 \\
F_{M,p}(s) \quad &\mathrm{for}\, s>0.\\
\end{cases}
\end{equation}
Let $R \in \Z$. Then 
\begin{equation}\label{prodform}
\lim_{t \to \infty}\Pb(Y_{M}(t)\geq -R)=\begin{cases} F_{M,p}(0)&\mathrm{for} \,  R \geq M \\
F_{M,p}(0)F_{M-R,p}(0) &\mathrm{for} \,  R < M.
\end{cases}
\end{equation}
\end{tthm}
Since  $\lim_{R\to - \infty }F_{M-R,p}(0)=0,$ (see Proposition \ref{propprop}), the product structure \eqref{prodform} interpolates between the two values of the cutoff \eqref{cutoff} if one sends $s\searrow0$ in \eqref{cutoff}.
Furthermore, $x_{M}$ does not enter the shock with probability $1-F_{M,p}(0)$ and this probability goes to $1$ as $M\to \infty$.

The following convergence result is the starting point  for the present work.
 For PASEP, the following Theorem  was shown in \cite{TW08b}, Theorem 2, for TASEP, the result follows e.g.  from \cite{GW90}, Corollary 3.3, see Remark 3.1 of \cite{GW90} for further references . An alternative characterization of the limit \eqref{ASEPlimit}
 was given in \cite{BO17}, Proposition 11.1.
\begin{tthm}[Theorem 2 in \cite{TW08b},  Corollary 3.3 in \cite{GW90}]\label{convthm}
Consider ASEP with step initial data $x_{n}^{\mathrm{step}}(0)=-n,n\geq 1$.
Then  for every fixed $M\geq 1$ we have that 
\begin{equation}\label{ASEPlimit}
\lim_{t \to \infty}\Pb\left(x_{M}^{\mathrm{step}}(t)\geq (p-q)(t-st^{1/2})\right)=F_{M,p}(s).
\end{equation}
\end{tthm}

\section{Preparatory  results}

Here we give two results needed later on. The first proposition asserts that for an ASEP started from an initial data where $\Z_{\geq 0}$ is initially occupied (but particles still have drift to the right), particles do not move 
too far to the left during a time interval $[0,t]$. 
\begin{prop}\label{block}
Consider  ASEP with reversed step  initial data $x_{-n}^{\mathrm{-step}}(0)=n,n\geq 0.$ Then  there is a $t_0$ such that for $t>t_{0}$ and constants $C_{1},C_{2}$ (which depend on $p$) we have
\begin{equation}\label{bms}
\Pb\left( \inf_{0\leq \ell \leq t} x_{0}^{\mathrm{-step}}(\ell)<-t^{1/4}\right)\leq C_{1}e^{-C_{2}t^{1/8}}.
\end{equation}
\end{prop}
\begin{proof}
We prove the proposition by comparing the reversed step initial data $\eta^{-\mathrm{step}}=\mathbf{1}_{i\geq0}$ with an invariant blocking measure $\mu$.
The measure $\mu$ on $\{0,1\}^{\Z}$ is the product measure with marginals
\begin{equation}
\mu(\{\eta:\eta(i)=1\})=\frac{c(p/q)^{i}}{1+c(p/q)^{i}},
\end{equation}
here $c>0$ is a free parameter we choose at the end of the proof. It is well known that $\mu$ is invariant for ASEP \cite{Lig76}.

Let $(\eta^{\mathrm{block}}_{s})_{s\geq 0}$  be the ASEP started from the initial distribution $\mu$, and denote by $x_{0}^{\mathrm{block}}(s)$ the position of the left most particle of $\eta^{\mathrm{block}}_{s}$.
Let $(\eta^{\mathrm{-step}}_{s})_{s\geq 0}$  be the ASEP started from the reversed step initial data  $\eta^{-\mathrm{step}}.$

Let us start by bounding for any fixed  $0\leq \ell \leq t$ 
\begin{equation}\label{equation}
\Pb\left(  x_{0}^{\mathrm{-step}}(\ell)<-t^{1/4}/2\right)\leq \frac{4}{c}\frac{1}{1-q/p}+4c\left(\frac{q}{p}\right)^{t^{1/4}/2}\frac{1}{1-q/p}.
\end{equation}
Consider the partial order on $\{0,1\}^{\Z}$ given by 
\begin{equation}
\eta \leq \eta^{\prime}\Leftrightarrow \eta(i)\leq\eta^{\prime}(i) \,\mathrm{\,for\,\,all\,\,}i\in \Z
\end{equation}
and use $\eta \not \leq \eta^{\prime}$ as short hand for the statement  that $\eta \leq \eta^{\prime}$ does not hold.
We can now bound
\begin{equation}
\begin{aligned}\label{un}
\Pb\left( x_{0}^{\mathrm{-step}}(\ell)<-t^{1/4}/2\right)&\leq \Pb\left( \{ x_{0}^{\mathrm{-step}}(\ell)<-t^{1/4}/2\}\cap \{\eta_{\ell}^{\mathrm{block}}\geq \eta_{\ell}^{\mathrm{-step}}\}\right)
\\&+\Pb(\eta_{\ell}^{\mathrm{block}}\not \geq \eta_{\ell}^{\mathrm{-step}}).
\end{aligned}
\end{equation}
Furthermore, we have\begin{equation}
\begin{aligned}\label{deux}
\Pb\left( \{ x_{0}^{\mathrm{-step}}(\ell)<-t^{1/4}/2\}\cap \{\eta_{\ell}^{\mathrm{block}}\geq \eta_{\ell}^{\mathrm{-step}}\}\right)&\leq \Pb\left(  x_{0}^{\mathrm{block}}(\ell)<-t^{1/4}/2\right)
\\&=\Pb\left(  x_{0}^{\mathrm{block}}(0)<-t^{1/4}/2\right)
\end{aligned}\end{equation}
where the identity \eqref{deux} follows from the invariance of $\mu$.

By attractivity of ASEP,
\begin{equation}\label{trois}
\Pb(\eta_{\ell}^{\mathrm{block}}\not \geq \eta_{\ell}^{\mathrm{-step}})\leq \Pb(\eta_{0}^{\mathrm{block}}\not \geq \eta^{\mathrm{-step}}).
\end{equation}
Using the simple estimates $\log(1+\varepsilon)\leq 2\varepsilon$ and $\exp(-\varepsilon)\geq 1-2\varepsilon$ for $\varepsilon\geq 0$
we obtain
\begin{equation}
\begin{aligned}\label{quatre}
\Pb(\eta_{0}^{\mathrm{block}}\not \geq \eta^{\mathrm{-step}})&=1-\prod_{i=0}^{\infty}\frac{1}{1+(q/p)^{i}/c}
\\&=1-\exp\left(-\sum_{i=0}^{\infty}\log(1+(q/p)^{i}/c)\right)
\\&\leq 1-\exp\left(-\sum_{i=0}^{\infty} 2(q/p)^{i}/c\right)
\\&=1-\exp\left(- 2/(c(1-q/p))\right)
\\&\leq \frac{4}{c}\frac{1}{1-q/p}.
\end{aligned}
\end{equation}
By a very similar computation we obtain
\begin{equation}\label{cinq}
\begin{aligned}
\Pb\left(  x_{0}^{\mathrm{block}}(0)<-t^{1/4}/2\right)&=1-\prod_{i=0}^{\infty}\frac{1}{1+c(q/p)^{t^{1/4}/2}(q/p)^{i}}
\\&\leq 4c\left(\frac{q}{p}\right)^{t^{1/4}/2}\frac{1}{1-q/p}.
\end{aligned}
\end{equation}

This proves \eqref{equation} by combining the inequalities \eqref{un},\eqref{deux},\eqref{trois},\eqref{quatre}and \eqref{cinq}.

Since \eqref{equation} does not depend on $\ell$, we obtain 
\begin{equation}\label{equation2}
\Pb\left( \bigcup_{\ell=1,2,\ldots,\lfloor t\rfloor } \{x_{0}^{\mathrm{-step}}(\ell)<-t^{1/4}/2\}\right)\leq t\left( \frac{4}{c}\frac{1}{1-q/p}+4c\left(\frac{q}{p}\right)^{t^{1/4}/2}\frac{1}{1-q/p}\right)
\end{equation}
Note further that for the event 
\begin{equation}\label{event}
\bigcap_{\ell=1,2,\ldots,\lfloor t \rfloor } \{x_{0}^{\mathrm{-step}}(\ell)\geq -t^{1/4}/2\}\cap\{ \inf_{0\leq \ell \leq t} x_{0}^{\mathrm{-step}}(\ell)<-t^{1/4}\}
\end{equation}
to hold, $x_{0}^{\mathrm{-step}}$ would need to make $t^{1/4}/2$ jumps to the left in a time interval $[\ell,\ell+1],\ell=0,\ldots,t-1$. 
For any fixed time interval $[\ell,\ell+1]$ the probability that  $x_{0}^{\mathrm{-step}}$ makes at least  $k$ jumps to the left is bounded by the probability that a rate $q$ Poisson process makes at least $k$ jumps in a unit time interval.
In particular, the probability that  $x_{0}^{\mathrm{-step}}$ makes $t^{1/4}/2$ jumps to the left during $[\ell,\ell+1]$ may be  bounded by $e^{-t^{1/4}/2}$ (for optimal bounds, see e.g.  \cite{G87}). Since there are $t$ such intervals, we see that the probability of the event \eqref{event} is bounded by $te^{-t^{1/4}/2}$.
So in total we obtain 
\begin{align}
\Pb( \inf_{0\leq \ell \leq t} x_{0}^{\mathrm{-step}}(\ell)<-t^{1/4})\leq &   \Pb\left( \bigcup_{\ell=1,2,\ldots,t} \{x_{0}^{\mathrm{-step}}(\ell)<-t^{1/4}/2\}\right)\\&+\Pb\left(\bigcap_{\ell=1,2,\ldots,t} \{x_{0}^{\mathrm{-step}}(\ell)\geq -t^{1/4}/2\}\cap\{ \inf_{0\leq \ell \leq t} x_{0}^{\mathrm{-step}}(\ell)<-t^{1/4}\}\right)
\\&\label{ja!}\leq   t\left( \frac{4}{c}\frac{1}{1-q/p}+4c\left(\frac{q}{p}\right)^{t^{1/4}/2}\frac{1}{1-q/p}+e^{-t^{1/4}/2}\right).
\end{align}
Choosing $c=(p/q)^{t^{1/8}}$ in \eqref{ja!} we obtain \eqref{bms} for $t $ sufficiently large.
\end{proof}

Finally, we need the following result  about the distribution function $F_{M,p}$.
\begin{prop}\label{propprop}

For any fixed  $s \in \R$  we have 
\begin{equation}\label{two1}
\lim_{M \to \infty}F_{M,p}(s)=0.
\end{equation}

\end{prop}

\begin{proof}
For TASEP, this is obvious e.g. from the known convergence of $F_{M,1}$ to $F_{\GUE}$. 
For PASEP, instead of using definition \eqref{def1}, we use the alternative description provided in \cite{BO17}.
By Proposition 11.1 of \cite{BO17},   we have that for ASEP started from step initial data
\begin{equation}
\chi_{s,t}=\#\{\mathrm{particles \,\, to \,\,the \,\,right\,\, of\,} (p-q)(t-s t^{1/2}) \, \mathrm{at \,\, time \,\, } t\}
\end{equation}
converges, as $t\to \infty$ to a $\Z_{\geq 0}$ valued random variable $\xi_{s}.$ Using Theorem \ref{convthm}, we thus get
\begin{equation}
\lim_{M \to \infty}F_{M,p}(s)=\lim_{M\to \infty}\lim_{t \to \infty}\Pb(\chi_{s,t}\geq M)=\lim_{M\to \infty}\Pb(\xi_{s}\geq M)=0.
\end{equation} 
\end{proof}
\section{Asymptotic Independence and Slow Decorelation}
In this section we prove two Propositions essential to the proof of the main result Theorem \ref{main}. Proposition \ref{indepprop} establishes the asymptotic independence of $x_{M}^{A}(t-t^{\nu}),x_{M}^{B}(t)$ for $\nu\in (1/2,1)$, whereas Proposition \ref{prop0} controls the difference between $x_{M}^{A}(t-t^{\nu})$ and $x_{M}^{A}(t)$. 
\begin{prop}\label{indepprop}
Let $\nu \in (1/2,1),R \in \Z, C \in \R.$ Then 
\begin{equation*}
\lim_{t\to \infty}\Pb(\min\{x_{M}^{A}(t-t^{\nu})+(p-q)(t^{\nu}+Ct^{1/2}),x_{M}^{B}(t)\}\geq -R)=\begin{cases} F_{M,p}(C)&\mathrm{for} \,  R \geq M \\
F_{M,p}(C)F_{M-R,p}(0) &\mathrm{for} \,  R < M.
\end{cases}
\end{equation*} 
\end{prop}
\begin{proof}

 Consider first the case $R<M$. Define the collection of holes
 \begin{equation}\label{holes}
H_{n}^{B}(0)=n+\lfloor (p-q)t\rfloor, \quad n\geq 1.
\end{equation}
Note the $H_{n}^{B}$ perform an ASEP with (shifted) step initial data, where the holes jump to the right with probability $q<1/2$ and to the left with probability $p=1-q$.
Then 
\begin{equation}
\{H_{M-R}^{B}(t)<-R\}=\{x_{M}^{B}(t)\geq -R\}.
\end{equation}
We note
\begin{equation}
\begin{aligned}\label{need}
&\lim_{t \to \infty}\Pb (x_{M}^{A}(t-t^{\nu})+(p-q)(t^{\nu}+Ct^{1/2})\geq -R)=F_{M,p}(C)
\\&\lim_{t \to \infty}\Pb (H_{M-R}^{B}(t)<-R)=F_{M-R,p}(0).
\end{aligned}
\end{equation}
Define now $\tilde{\eta}^{A}_0=\eta_{0}^{A}$ . Let $0<\varepsilon<\nu-1/2$. 
Graphically construct $(\tilde{\eta}^{A}_s)_{s\geq 0}$ just like $(\eta^{A}_s)_{s\geq 0},$ using the same Poisson processes,  with the difference that all jumps
in the space-time region
\begin{align}
\{(i,s)\in \Z\times \R_{+}:i\geq-(p-q)t^{\nu}+t^{1/2+\varepsilon}, 0\leq s\leq t-t^{\nu}\}
\end{align}
are surpressed. Denote by $\tilde{x}_{M}^{A}(t-t^{\nu})$ the position of the $Mth$  particle (counted from right to left) of $\tilde{\eta}^{A}_{t-t^{\nu}}$.

Likewise, define  $\tilde{\eta}^{B}_0=\eta_{0}^{B}$ . 
Graphically construct $(\tilde{\eta}^{B}_s)_{s\geq 0}$ just like $(\eta^{B}_s)_{s\geq 0},$ using the same Poisson processes,  with the difference that all jumps
in the space-time region
\begin{align}
\{(i,s)\in \Z\times \R_{+}:i\leq-t^{1/2+\varepsilon}, 0\leq s\leq t\}
\end{align}
are surpressed. Denote by $\tilde{H}_{M-R}^{B}(t)$ the position of the $(M-R)$th  hole (counted from left to right) of $\tilde{\eta}^{B}_{t}$.

Then, $\tilde{H}_{M-R}^{B}(t),\tilde{x}_{M}^{A}(t-t^{\nu})$ are independent random variables.
Define the event
\begin{align}
G_{t}=\{\tilde{H}_{M-R}^{B}(t)\neq H_{M-R}^{B}(t)\}\cup \{ \tilde{x}_{M}^{A}(t-t^{\nu})\neq x_{M}^{A}(t-t^{\nu})\}.
\end{align} 
We show
\begin{align}\label{ja}
\lim_{t \to \infty}\Pb(G_{t})=0.
\end{align}
To see this,
note
\begin{align}
\{\tilde{x}_{M}^{A}(t-t^{\nu})\neq x_{M}^{A}(t-t^{\nu})\}\subseteq \{\sup_{0\leq \ell\leq t-t^{\nu}}x_{1}^{A}(\ell)\geq -(p-q)t^{\nu}+t^{1/2+\varepsilon}\}.
\end{align}
By Theorem \ref{convthm},
\begin{align}
\lim_{t \to \infty}\Pb (x_{1}^{A}(t-t^{\nu})\geq -(p-q)t^{\nu}+t^{1/2+\varepsilon}/2)\leq \lim_{s\to -\infty}F_{M,p}(s)=0
\end{align}
so that 
\begin{align}
&\lim_{t\to \infty}\Pb\left(\sup_{0\leq \ell\leq t-t^{\nu}}x_{1}^{A}(\ell)\geq -(p-q)t^{\nu}+t^{1/2+\varepsilon}\right)
\\&=\lim_{t\to \infty} \Pb\left (\sup_{0\leq \ell\leq t-t^{\nu}}x_{1}^{A}(\ell)\geq -(p-q)t^{\nu}+t^{1/2+\varepsilon},\,  x_{1}^{A}(t-t^{\nu})\leq -(p-q)t^{\nu}+t^{1/2+\varepsilon}/2 \right)
\end{align}
Start an ASEP from the initial data 
\begin{equation}
\hat{\eta}_{0}=\mathbf{1}_{i\geq -(p-q)t^{\nu}+t^{1/2+\varepsilon}}
\end{equation}
which is a shifted reversed step initial data. Denote by $\hat{x}_{0}(\ell)$ the position of the left most particle of $\hat{\eta}_{\ell}$  at time $\ell$.

Now on the event $\{\sup_{0\leq \ell\leq t-t^{\nu}}x_{1}^{A}(\ell)\geq -(p-q)t^{\nu}+t^{1/2+\varepsilon}\}$ there is a $\lambda_{1}\in [0,t-t^{\nu}]$ such that
$x_{1}^{A}(\lambda_{1})\geq  -(p-q)t^{\nu}+t^{1/2+\varepsilon}$. In particular, then 
\begin{equation}
x_{1}^{A}(\lambda_{1})\geq \hat{x}_{0}(\lambda_{1})
\end{equation}
which implies
\begin{equation}
x_{1}^{A}(\ell)\geq \hat{x}_{0}(\ell), \lambda_{1}\leq \ell\leq t-t^{\nu}.
\end{equation}
So we have shown
\begin{equation}
\{\sup_{0\leq \ell\leq t-t^{\nu}}x_{1}^{A}(\ell)\geq -(p-q)t^{\nu}+t^{1/2+\varepsilon}\}\subseteq \{x_{1}^{A}(t-t^{\nu})\geq \hat{x}_{0}(t-t^{\nu})\}.
\end{equation}
Thus
\begin{align*}
& \lim_{t\to \infty} \Pb\left (\sup_{0\leq \ell\leq t-t^{\nu}}x_{1}^{A}(\ell)\geq -(p-q)t^{\nu}+t^{1/2+\varepsilon},\,  x_{1}^{A}(t-t^{\nu})\leq -(p-q)t^{\nu}+t^{1/2+\varepsilon}/2 \right)
\\&\leq  \lim_{t\to \infty} \Pb\left ( \hat{x}_{0}(t-t^{\nu})\leq -(p-q)t^{\nu}+t^{1/2+\varepsilon}/2 \right)
\\&\leq  \lim_{t\to \infty} \Pb\left ( \hat{x}_{0}(t-t^{\nu})+(p-q)t^{\nu}-t^{1/2+\varepsilon}\leq - t^{1/2+\varepsilon}/2 \right)
\\&=  \lim_{t\to \infty} \Pb\left ( x_{0}^{\mathrm{-step}}(t-t^{\nu})\leq - t^{1/2+\varepsilon}/2 \right)
\\&=0
\end{align*}
where in the last step we used Proposition \ref{block}. So we have shown 
\begin{equation}
\lim_{t \to \infty}\Pb(\tilde{x}_{M}^{A}(t-t^{\nu})\neq x_{M}^{A}(t-t^{\nu}))=0.
\end{equation}
The proof of 
\begin{equation}
\lim_{t \to \infty}\Pb(\tilde{H}_{M-R}^{B}(t)\neq H_{M-R}^{B}(t))=0
\end{equation}
is almost identical, one notes 
\begin{align}
\{(\tilde{H}_{M-R}^{B}(t)\neq H_{M-R}^{B}(t)\}\subseteq \{\inf_{0\leq \ell\leq t}H_{1}^{B}(\ell)\leq -t^{1/2+\varepsilon}\}.
\end{align}
and deduces  $\lim_{t\to \infty}\Pb(\{\inf_{0\leq \ell\leq t}H_{1}^{B}(\ell)\leq -t^{1/2+\varepsilon}\})=0$ from 
\begin{align}
\lim_{t \to \infty}\Pb (H_{1}^{B}(t)\leq -t^{1/2+\varepsilon}/2)=0
\end{align}

So we have shown  \eqref{ja}. We conclude by computing
\begin{align*}
&\lim_{t\to \infty}\Pb(\min\{x_{M}^{A}(t-t^{\nu})+(p-q)(t^{\nu}+Ct^{1/2}),x_{M}^{B}(t)\}\geq -R)
\\&=\lim_{t\to \infty}\Pb(x_{M}^{A}(t-t^{\nu})+(p-q)(t^{\nu}+Ct^{1/2})\geq -R, H_{M-R}^{B}(t)<-R)
\\&=\lim_{t\to \infty}\Pb(\{x_{M}^{A}(t-t^{\nu})+(p-q)(t^{\nu}+Ct^{1/2})\geq -R\}\cap \{ H_{M-R}^{B}(t)<-R\}\cap G_{t}^{c})
\\&=\lim_{t\to \infty}\Pb(\{\tilde{x}_{M}^{A}(t-t^{\nu})+(p-q)(t^{\nu}+Ct^{1/2})\geq -R\}\cap \{ \tilde{H}_{M-R}^{B}(t)<-R\}\cap G_{t}^{c})
\\&=\lim_{t\to \infty}\Pb(\{\tilde{x}_{M}^{A}(t-t^{\nu})+(p-q)(t^{\nu}+Ct^{1/2})\geq -R\}\cap \{ \tilde{H}_{M-R}^{B}(t)<-R\})
\\&= \lim_{t\to \infty}\Pb(\tilde{x}_{M}^{A}(t-t^{\nu})+(p-q)(t^{\nu}+Ct^{1/2})\geq -R)\Pb(\tilde{H}_{M-R}^{B}(t)<-R)
\\&= \lim_{t\to \infty}\Pb(\{\tilde{x}_{M}^{A}(t-t^{\nu})+(p-q)(t^{\nu}+Ct^{1/2} \})\geq -R\}\cap G_{t}^{c})  \Pb(\{\tilde{H}_{M-R}^{B}(t)<-R\}\cap G_{t}^{c})
\\&= \lim_{t\to \infty}\Pb(\{x_{M}^{A}(t-t^{\nu})+(p-q)(t^{\nu}+Ct^{1/2} \})\geq -R\}\cap G_{t}^{c})  \Pb(\{H_{M-R}^{B}(t)<-R\}\cap G_{t}^{c})
\\&= \lim_{t\to \infty}\Pb(\{x_{M}^{A}(t-t^{\nu})+(p-q)(t^{\nu}+Ct^{1/2} \})\geq -R\})  \Pb(\{H_{M-R}^{B}(t)<-R\})
\\&=F_{M,p}(C)F_{M-R,p}(0),
\end{align*}
where for the last identity we used \eqref{need}.

Finally, consider the case $R\geq M$. Note that then $\Pb(x_{M}^{B}(t)\geq -R)=1$ and thus 
\begin{align*}&\lim_{t\to \infty}\Pb(\min\{x_{M}^{A}(t-t^{\nu})+(p-q)(t^{\nu}+Ct^{1/2}),x_{M}^{B}(t)\}\geq -R)\\&=\lim_{t\to \infty}\Pb(x_{M}^{A}(t-t^{\nu})+(p-q)(t^{\nu}+Ct^{1/2})\geq -R)\\&=F_{M,p}(C).
\end{align*}
\end{proof}
We recall the following elementary Lemma. We denote by $"\Rightarrow"$ convergence in distribution.
\begin{lem}\label{elemlem}
Let $(X_{n})_{n\geq 1}, (\tilde{X}_{n})_{n\geq 1}$ be sequences of random variables such that $X_{n}\geq \tilde{X}_{n}$.Let $X_{n } \Rightarrow D, \tilde{X}_{n}\Rightarrow D,$ where $D$ is a probability distribution. Then $X_{n}- \tilde{X}_{n}\Rightarrow0$.
\end{lem}

Next we prove the following slow-decorrelation type statement.
\begin{prop}\label{prop0}
Let $\nu \in (1/2,1)$ and $\varepsilon>0$. Then  
\begin{equation}
\lim_{t\to \infty}\Pb\left(\left|x_{M}^{A}(t)-x_{M}^{A}(t-t^{\nu})-(p-q)t^{\nu}\right|\geq \varepsilon t^{1/2}\right)=0.
\end{equation} 
\end{prop}

\begin{proof}
Consider an ASEP with step initial data which starts at time $t-t^{\nu}$ and has its rightmost particle at position $x_{M}^{A}(t-t^{\nu})$:
Set $\tilde{\eta}_{s}=\tilde{\eta}_{t-t^{\nu}+s},s\leq t^{\nu},$ and $\tilde{\eta}_{t-t^{\nu}}=\mathbf{1}_{i\leq x_{M}^{A}(t-t^{\nu})}$.
Denote by $\tilde{x}_{1}(t^{\nu})$ the position  of the right most particle of  $\tilde{\eta}_{t^{\nu}}$.
Then we have 
\begin{equation} \label{upperbound}
x_{M}^{A}(t)\leq x_{M}^{A}(t-t^{\nu})+\tilde{x}_{1}(t^{\nu})- x_{M}^{A}(t-t^{\nu})
\end{equation}
Now
\begin{equation}\label{dashier}\tilde{x}_{1}(t^{\nu})- x_{M}^{A}(t-t^{\nu})=^{d}x_{1}^{\mathrm{step}}(t^{\nu})\end{equation}
where $=^{d}$ denotes equality in distribution and $x_{1}^{\mathrm{step}}(t^{\nu})$ is the position at time $t^{\nu}$ of the right most particle 
in ASEP started with step initial data $x_{n}^{\mathrm{step}}(0)=-n,n\geq 1$.
Now by Theorem \ref{convthm} we have in particular that $((x_{1}^{\mathrm{step}}(t^{\nu})-(p-q)t^{\nu})t^{-\nu/2}),t\geq 0$ is tight,
which together with  \eqref{dashier} implies
\begin{equation}\label{eins}
\lim_{t \to \infty} \Pb\left(|\tilde{x}_{1}(t^{\nu})- x_{M}^{A}(t-t^{\nu})-(p-q)t^{\nu} |t^{-1/2}\geq \varepsilon/2\right)=0.
\end{equation}

Now by  Theorem \ref{convthm}
\begin{align}\label{zwei}
&\frac{x_{M}^{A}(t-t^{\nu})+(p-q)t^{\nu}}{t^{1/2}(p-q)}\Rightarrow F_{M,p}
\quad  \frac{x_{M}^{A}(t)}{t^{1/2}(p-q)}\Rightarrow F_{M,p}
\end{align}
So by \eqref{eins},
\begin{align}
&\frac{x_{M}^{A}(t-t^{\nu})+(p-q)t^{\nu}}{t^{1/2}(p-q)}+      \frac{\tilde{x}_{1}(t^{\nu})-x_{M}^{A}(t-t^{\nu})-(p-q)t^{\nu}}{t^{1/2}(p-q)}    \Rightarrow F_{M,p}
\end{align}
Thus we can apply Lemma \ref{elemlem} to \eqref{upperbound}, 
which then implies 
\begin{align}
\frac{x_{M}^{A}(t-t^{\nu})+(p-q)t^{\nu}}{t^{1/2}(p-q)}- \frac{x_{M}^{A}(t)}{t^{1/2}(p-q)} \Rightarrow 0,
\end{align}
using \eqref{eins}.
\end{proof}

\section{Proof of Theorem \ref{main}}

Here we combine all previously given results to prove Theorem \ref{main}.
\begin{proof}[Proof of Theorem \ref{main}]
We start by proving \eqref{cutoff} for $s<0$.
We show the stronger statement $\lim_{t\to \infty}\Pb(x_{M}^{B}(t)>-st^{1/2})=0$ for $s<0.$
Consider the collection of holes from \eqref{holes}.
Then 
\begin{equation}
\{H_{-st^{1/2}+M}^{B}(t)<-st^{1/2}\}=\{x_{M}^{B}(t)\geq -st^{1/2}\}.
\end{equation}
Now we have
\begin{equation}
\Pb(H_{-st^{1/2}+M}^{B}(t)<-st^{1/2})=\Pb(x^{\mathrm{step}}_{-st^{1/2}+M}>(p-q)t+st^{1/2}).
\end{equation} 
Note  \eqref{two1}  implies that for  $f:\R_{>0}\to \R_{>0}$  an increasing function such that $\lim_{t\to \infty}f(t)=\infty$
\begin{equation}\label{two}
\lim_{t\to \infty}\Pb(x_{\lfloor f(t)\rfloor }^{\mathrm{step}}(t)>(p-q)t-st^{1/2})=0.
\end{equation}
Thus 
\begin{equation}
\lim_{t \to \infty} \Pb(x_{M}^{B}(t)>-st^{1/2})=\lim_{ t \to \infty}\Pb(x^{\mathrm{step}}_{-st^{1/2}+M}>(p-q)t+st^{1/2})=0
\end{equation}
by \eqref{two}.

Next we prove \eqref{cutoff} for $s>0$. Denote by $x_{-n}^{-\mathrm{step}}(0)=n,n\geq 0$ the reversed step initial data, and construct $(x_{-n}^{-\mathrm{step}}(\ell))_{\ell\geq 0, n\geq 0}$ on the same probability space as  $x_{M}(t),x_{M}^{A}(t),x_{M}^{B}(t)$.
By  Theorem \ref{convthm} and since $x_{M}(t)\leq x_{M}^{A}(t)$ we have
\begin{align}
F_{M,p}(s)=&\lim_{t \to \infty}\Pb(x_{M}^{A}(t)\geq  -(p-q)st^{1/2},x_{M}(t)< -(p-q)st^{1/2})
\\&+\lim_{t \to \infty}\Pb(x_{M}^{A}(t)\geq  -(p-q)st^{1/2},x_{M}(t)\geq  -(p-q)st^{1/2})
\\& =\lim_{t \to \infty}\Pb(x_{M}^{A}(t)\geq  -(p-q)st^{1/2},x_{M}(t)< -(p-q)st^{1/2})
\\&+\lim_{t \to \infty}\Pb(x_{M}(t)\geq  -(p-q)st^{1/2}).
\end{align} 
It thus suffices to prove
\begin{equation}
\lim_{t \to \infty}\Pb(x_{M}(t)< -(p-q)st^{1/2}, x_{M}^{A}(t)\geq  -(p-q)st^{1/2})=0.
\end{equation} 


Define 
\begin{align}
&\tau_{0}=0
\\&\tau_{i}=\inf\{\ell: x_{i}(\ell)\neq x_{i}^{A}(\ell)\}, \, i \geq 1.
\end{align} 
We show
\begin{align}\label{omit}
\{x_{M}(t)\neq x_{M}(t)\}\subseteq \mathcal{B}=\{ 0=\tau_{0}<\tau_{1}<\cdots<\tau_{M}\leq t, x_{i-1}(\tau_{i}) -x_{i}(\tau_{i})=1,i=1,\ldots,M\}
\end{align}
To see \eqref{omit},  note $0<\tau_{M}\leq t$ on $\{x_{M}(t)\neq x_{M}^{A}(t)\}.$ Recall further $x_{M}(\ell)\leq x_{M}^{A}(\ell)$ for all $\ell\geq 0$. Then we have \begin{equation*}x_{M}(\tau_{M})\neq x_{M}^{A}(\tau_{M}),x_{M}(\tau_{M}^{-})= x_{M}^{A}(\tau_{M}^{-}).\end{equation*}
Now $x_{M}(\tau_{M}^{-})= x_{M}^{A}(\tau_{M}^{-})$ implies  $x_{M+1}(\tau_{M}^{-})= x_{M+1}^{A}(\tau_{M}^{-})$ [Asumme to the contrary $x_{M}(\tau_{M}^{-})= x_{M}^{A}(\tau_{M}^{-}), x_{M+1}(\tau_{M}^{-})\neq  x_{M+1}^{A}(\tau_{M}^{-})$ both hold.  Since  $x_{M+1}(\tau_{M}^{-})\neq  x_{M+1}^{A}(\tau_{M}^{-})$ is equivalent to 
$x_{M+1}(\tau_{M}^{-})<  x_{M+1}^{A}(\tau_{M}^{-}),$ we have that $x_{M+1}(\tau_{M}^{-})\neq  x_{M+1}^{A}(\tau_{M}^{-})$ implies  \begin{equation*}x_{M+1}(\tau_{M}^{-})<  x_{M+1}^{A}(\tau_{M}^{-})< x_{M}^{A}(\tau_{M}^{-})= x_{M}(\tau_{M}^{-}).\end{equation*}But this cannot happen since 
then $\eta_{\tau_{M}^{-}}^{A}(x_{M+1}^{A}(\tau_{M}^{-}))=1>\eta_{\tau_{M}^{-}}(x_{M+1}^{A}(\tau_{M}^{-}))$ in contradiction to $\eta_{t}^{A}\leq \eta_{t}$ for all $t$].
Now that $x_{M}(\tau_{M}^{-})= x_{M}^{A}(\tau_{M}^{-}), x_{M+1}(\tau_{M}^{-})= x_{M+1}^{A}(\tau_{M}^{-})$  hold implies that the only way the discrepency  $x_{M}(\tau_{M})\neq  x_{M}^{A}(\tau_{M})$
can be created is by a jump to the right of $x_{M}^{A}$ that $x_{M}$ does not make (the other possibility to create this discrepency would be by a jump of $x_{M}$ to the left that $x_{M}^{A}$ does not make, but since 
$x_{M+1}(\tau_{M}^{-})= x_{M+1}^{A}(\tau_{M}^{-}),$ $x_{M}$ and $x_{M}^{A}$ can only jump together to the left at time $\tau_{M}$). This shows that at time $\tau_{M}$ a jump of $x_{M}$ has been supressed by the presence of $x_{M-1}$.
Furthermore, $x_{M-1}(\tau_{M})<x_{M-1}^{A}(\tau_{M})$ showing that
\begin{equation}
0<\tau_{M-1}<\tau_{M}.
\end{equation}
 Repeating the preceeding argument, we see that at time $\tau_{M-1}$ a jump of $x_{M-1}$ was supressed by the presence of $x_{M-2}.$ Iteratively, we obtain $0<\tau_{1}< \cdots<\tau_{M} \leq t $ and  that at time $\tau_{i}$, a jump of $x_{i}$ is supressed by the presence of $x_{i-1}$, $i=1,\ldots,M$. In particular, \eqref{omit} holds.
 Define
 \begin{equation}
 \eta_{\tau_{i}}^{i}=\mathbf{1}_{\{i\geq x_{i}(\tau_{i})\}}
 \end{equation}
 and the ASEP
 \begin{equation}
 \eta_{s}^{i}= \eta_{\tau_{i}+s}^{i}
 \end{equation}
 Denote by $x_{0}^{i}(s)$ the position of the left most particle of $ \eta_{s}^{i}$ .
 Then
 \begin{equation}
 x_{0}^{i}(s)\leq  x_{i}(\tau_{i}+s)
 \end{equation}
 and  $(x_{0}^{i}(s)-x_{0}^{i}(0))_{s\geq 0}$ is just an ASEP started from reversed step initial data $\eta^{-\mathrm{step}}=\mathbf{1}_{\{i\geq 0 \}}$.
Now we have 
\begin{align*}
&\lim_{t \to \infty}\Pb(x_{M}(t)< -(p-q)st^{1/2}, x_{M}^{A}(t)\geq  -(p-q)st^{1/2})
\\&=\lim_{t \to \infty}\Pb(x_{M}(t)< -(p-q)st^{1/2}, x_{M}^{A}(t)\geq  -(p-q)st^{1/2},x_{M}(t)\neq x_{M}^{A}(t))
\\&\leq \lim_{t \to \infty}\Pb(\{x_{M}(t)< -(p-q)st^{1/2}\}\cap \mathcal{B} ).
\end{align*} 
We show
\begin{equation}\label{label}
\Pb(\{x_{M}(t)< -(M+1)t^{1/4}\}\cap \mathcal{B} )\leq (M+1)C_{1}e^{-C_{2}t^{1/8}}.
\end{equation}
Define the event 
\begin{equation}
E_{i}=\{\inf_{\tau_{i}\leq \ell \leq t}x_{i}(\ell) -x_{i-1}(\tau_{i}) <-t^{1/4}\}\cap\mathcal{B}
\end{equation}
We have 
\begin{equation}\label{bbb}
\Pb(E_{i})\leq C_{1}e^{-C_{2}t^{1/8}}
\end{equation}
since 
\begin{equation}
x_{i}(\tau_{i}+s) -x_{i-1}(\tau_{i}) \geq x_{0}^{i}(s)-x_{0}^{i}(0)-1
\end{equation} and  $x_{0}^{i}(s)-x_{0}^{i}(0),s\geq 0$ is just an ASEP started from reversed step initial data, so \eqref{bbb} follows from  Proposition \ref{block}.
Next we make the crucial observation
\begin{equation}
\{x_{M}(t)< -(p-q)st^{1/2}\}\cap \mathcal{B}\setminus(\cup_{i=0}^{M}E_{i})=\emptyset.
\end{equation}
This implies  
\begin{equation}
\Pb(\{x_{M}(t)< -(p-q)st^{1/2}\}\cap \mathcal{B} )\leq \Pb(\cup_{i=0}^{M}E_{i})\leq (M+1)C_{1}e^{-C_{2}t^{1/8}}.
\end{equation}
finishing the proof of \eqref{cutoff} for $s>0$.


Next we come to \eqref{prodform}.
We define the event 
\begin{equation}
A_{t}=\{ \left| x_{M}^{A}(t)-x_{M}^{A}(t-t^{\nu})-(p-q)t^{\nu}\right| \leq \varepsilon t^{1/2}\}
\end{equation}
Let first $R\geq M$.
Then 
\begin{align}
\lim_{t\to \infty} \Pb (Y_{M}\geq -R)\leq \lim_{t\to \infty} \Pb (\{Y_{M}\geq -R\}\cap A_{t})+\Pb(A_{t}^{c}).
\end{align}
Now 
\begin{align}
\{Y_{M}\geq -R\}\cap A_{t}&=\{\min\{x_{M}^{A}(t)-x_{M}^{A}(t-t^{\nu})-(p-q)t^{\nu}+x_{M}^{A}(t-t^{\nu})+(p-q)t^{\nu},x_{M}^{B}(t)\}\geq -R\}\cap A_{t}\\&\subseteq\{\min\{\varepsilon t^{1/2}+x_{M}^{A}(t-t^{\nu})+(p-q)t^{\nu},x_{M}^{B}(t)\}\geq -R\}\cap A_{t}
\end{align}
and likewise
\begin{align}
\{Y_{M}\geq -R\}\cap A_{t} \supseteq\{\min\{-\varepsilon t^{1/2}+x_{M}^{A}(t-t^{\nu})+(p-q)t^{\nu},x_{M}^{B}(t)\}\geq -R\}\cap A_{t}
\end{align}
Thus
\begin{equation}
\begin{aligned}\label{fin1}
\lim_{t\to \infty} \Pb (Y_{M}\geq -R)&\leq \lim_{t\to \infty} \Pb (\{Y_{M}\geq -R\}\cap A_{t})+\Pb(A_{t}^{c})
\\&\leq   \lim_{t\to \infty} \Pb(\{\min\{\varepsilon t^{1/2}+x_{M}^{A}(t-t^{\nu})+(p-q)t^{\nu},x_{M}^{B}(t)\}\geq -R\}\cap A_{t})+\Pb(A_{t}^{c})
\\& \leq  \lim_{t\to \infty} \Pb(\{\min\{\varepsilon t^{1/2}+x_{M}^{A}(t-t^{\nu})+(p-q)t^{\nu},x_{M}^{B}(t)\}\geq -R\})+\Pb(A_{t}^{c})
\\&=F_{M,p}(\varepsilon)
\end{aligned}
\end{equation}
where the last equality follows from Propositions \ref{indepprop} and \ref{prop0}.
Likewise,
\begin{equation}\label{fin2}
\begin{aligned}
\lim_{t\to \infty} \Pb (Y_{M}\geq -R)&\geq \lim_{t\to \infty} \Pb (\{Y_{M}\geq -R\}\cap A_{t})
\\&\geq   \lim_{t\to \infty} \Pb(\{\min\{-\varepsilon t^{1/2}+x_{M}^{A}(t-t^{\nu})+(p-q)t^{\nu},x_{M}^{B}(t)\}\geq -R\}\cap A_{t})
\\& \geq  \lim_{t\to \infty} \Pb(\{\min\{-\varepsilon t^{1/2}+x_{M}^{A}(t-t^{\nu})+(p-q)t^{\nu},x_{M}^{B}(t)\}\geq -R\})-\Pb(A_{t}^{c})
\\&=F_{M,p}(-\varepsilon)
\end{aligned}
\end{equation}
Since $\varepsilon>0$ is arbitrary, it follows
\begin{equation}
\lim_{t\to \infty} \Pb (Y_{M}\geq -R)=F_{M,p}(0).
\end{equation}
Let now $R<M.$ We have the same inequalities  as in \eqref{fin1},\eqref{fin2}, except that they now yield 
\begin{equation}
F_{M,p}(-\varepsilon)F_{M-R,p}(0)\leq \lim_{t\to \infty} \Pb (Y_{M}\geq -R)\leq F_{M,p}(\varepsilon)F_{M-R,p}(0),
\end{equation}
implying 
\begin{equation}
\lim_{t\to \infty} \Pb (Y_{M}\geq -R)=F_{M,p}(0)F_{M-R,p}(0).
\end{equation}
\end{proof}

\bibliographystyle{myhalpha}
\bibliography{Biblio.bib}

\end{document}